\documentclass[11pt,reqno]{amsart}
\usepackage{amsmath,amsthm,amscd,amsfonts,amssymb,color}
\usepackage{cite}
\usepackage[mathscr]{eucal}
\usepackage[bookmarksnumbered,colorlinks,plainpages]{hyperref}
\newtheorem{theorem}{Theorem}[section]

\newtheorem{corollary}[theorem]{Corollary}
\theoremstyle{definition}
\newtheorem{definition}[theorem]{Definition}
\newtheorem{example}[theorem]{Example}
\theoremstyle{remark}

\numberwithin{equation}{section}
\def\DJ{\leavevmode\setbox0=\hbox{D}\kern0pt\rlap
 {\kern.04em\raise.188\ht0\hbox{-}}D}

\begin{document}

%\title[An application to a thermostat model via fixed point results in Banach spaces]{An application to a thermostat model via fixed point results in Banach spaces}
%\title[Altering distance functions in ... thermostat model]{Altering distance functions in fixed point results with application to thermostat model}
\title[Common solution \dots using fixed point results]{Common solution to a pair of non-linear matrix equations via fixed point results}
\author[H.\ Garai, L.K. \ Dey,]
{Hiranmoy Garai$^{1}$,  Lakshmi Kanta Dey$^{2}$,}

%\thanks{}
\address{{$^{1}$\,} Hiranmoy Garai,
                    Department of Mathematics,
                    National Institute of Technology
                    Durgapur, India.}
                    \email{hiran.garai24@gmail.com}
\address{{$^{2}$\,} Lakshmi Kanta Dey,
                    Department of Mathematics,
                    National Institute of Technology
                    Durgapur, India.}
                    \email{lakshmikdey@yahoo.co.in}
%\address{{$^{3}$\,} Ankush Chanda,
%                    Department of Mathematics,
%                    National Institute of Technology
%                    Durgapur, India.}
%                    \email{ankushchanda8@gmail.com}
%                    

\subjclass{$47H10$, $54H25$.}
\keywords{Altering distance function, fixed point, Banach space, double sequence, thermostat model.}

%\subjclass[2010]{$47$H$10$, $54$H$25$.}
%\keywords{Metric space; boundedly compact set; $T$-orbitally compact space; Kannan type mapping, fixed point property}
\begin{abstract}
In this article, we propose an idea to develop some sufficient conditions for the existence and uniqueness of a positive definite common solution to a pair of non-linear matrix equations. To proceed this, we present some interesting common fixed point results involving couple of altering distance functions along with some other control functions in Banach spaces. Based on these results, we deduce some desired sufficient conditions for the existence and uniqueness  of a positive definite common solution to the said pair of non-linear matrix equations. We  point out a probable applicable area of our findings.
%The motive behind this manuscript is to set up the existence and uniqueness of a positive solution to a fractional thermostat model for certain values of the parameter $\lambda>0$. We accomplish sufficient
%conditions for the existence of a positive solution to the model, and afterwards formulate a non-trivial example to authenticate the grounds of our obtained results. Our findings are based on certain fixed point results for contractions depending on a couple of altering distance functions $\phi$ and $\psi$ in the setting of Banach spaces.
\end{abstract}
 
\maketitle

\setcounter{page}{1}

\centerline{}

\centerline{}

\section{\bf Introduction and Preliminaries}
Matrix equations are broadly used in many scientific areas such as  physics, engineering, computer sciences,  information sciences,  system and control theory etc. The research on solvability conditions and general solutions of different types of linear and non-linear matrix equations is  one of the central importances in matrix theory and its applications. One of the main focus in this topic is to provide solvability conditions, special and general solutions and structural representations of some important matrix equations (linear or non-linear).   One such   non-linear matrix equation is of  the following form
\begin{equation}\label{me1}
X=Q + \displaystyle \sum_{i=1}^{m} {A_i}^*F(X)A_i
\end{equation}
where $Q$ is a Hermitian positive definite $n \times n$ matrix and  $A_i$'s are arbitrary $n \times n$ matrices, $F$ is a map from the set of all  $n \times n$ Hermitian matrices into itself.  This type of matrix equation frequently occurs in  mathematical modeling of many real life problems including control theory, ladder networks, dynamic programming, stochastic filtering, image processing, computer graphics, medical imaging solution methods etc. [see  \cite{W1,TW} and references therein]. In this context,   finding some sufficient conditions to ensure the existence and uniqueness of a positive definite solution $X$ of Equation \eqref{me1} plays a crucial role.  Some  more related works  about matrix equations are discussed in \cite {RR,ER,BS} and the references cited therein. 

It is interesting to note that in the context of non-interacting control theory with or without internal stability, a special case may arise that a composite non-linear system having two different  kinds of control input and measurement output needs a solution which will satisfy both the inputs and outputs. Such type of problems may be handled by finding a common solution to some family of non-linear matrix equations. So far there are no results associated with the existence of common solution(s) to the above type of equations.  Motivated by these facts our main aim of this paper is to study the common solutions of the following pair of matrix equations:
\begin{align}
X=Q_1 + \displaystyle \sum_{i=1}^{m} {A_i}^*F(X)A_i \label{me2}\\
X=Q_2 + \displaystyle \sum_{i=1}^{m} {A_i}^*G(X)A_i \label{me3}
\end{align}
where $Q_1,Q_2$ are Hermitian positive definite $n \times n$ matrices and  $A_i$'s are arbitrary $n \times n$ matrices, and $F,G$ are maps from the set of all  $n \times n$ Hermitian matrices into itself. Also they are continuous functions with respect to some suitable norm on the set of all Hermitian matrices. 

In order to tackle our main problem we derive a beautiful correlation between the existence of common solution to the pair of non-linear matrix equations \eqref {me2} and  \eqref {me3} and the existence of common fixed point of a pair of suitable Banach space valued self-mappings. Consequently we establish some interesting common fixed point results for a pair of self-mappings on a Banach space by using a control function, known as altering distance function. On the basis of our common fixed point result we deduce some sufficient conditions for existence and uniqueness of a common positive definite solution of the above mentioned matrix equations.

%For further details of altering distance functions, see \cite {R1,DC,PM,N1,SB}. Using our derived common fixed point results we have construct some conditions sufficient for the existence  and uniqueness of positive definite common solution to the pair of matrix equations \eqref {me2} and  \eqref {me3}.

Throughout this article we consider matrices over \textit{complex numbers} and use the notations $M(n),H(n),P(n),\overline{P(n)}$ as the set of all $n\times n$ matrices, $n\times n$ Hermitian matrices, $n\times n$ Hermitian positive definite matrices and  $n\times n$  Hermitian positive semi-definite matrices respectively. 
  %In this manuscript, on the basis of our common fixed point result we construct some sufficient conditions for existence and uniqueness of common positive definite solution of the above mentioned matrix equations. 

\section{\bf The Common Fixed Point Results}
In the beginning of this section we prove some common fixed point results involving a pair of altering distance functions. The concept of altering distance function was first initiated by Khan et al. \cite{KSS} and they define such kind of control functions as follows:
\begin{definition} \cite{KSS}
A function $\varphi:[0,\infty)\to [0,\infty)$ is called an altering distance function if it satisfies the following conditions:
\begin{enumerate}
\item[$(i)$] 
$\varphi$ is monotone increasing and continuous;
\item[$(ii)$] 
$\varphi(t)=0$ if and only if $t=0$.
\end{enumerate}
\end{definition}
Since the last few years fixed point theory via altering distance functions is one of the foremost subject in the literature of fixed point theory. Many mathematicians have proved several fixed point results utilizing the notion of this type of control function which extend, unify and generalize several well known results related to fixed points, see \cite {R1,DC,PM,N1,SB}. 

We now prove a common fixed point result concerning a couple of altering distance functions and one more continuous control function.
\begin{theorem} \label{t1}
Let $X$ be a Banach space and $C$ be a closed subset of $X$. Let $f,g :C \to C$ be two continuous functions such that
\begin{equation} \label{e1}
\phi(\|f(x)-g(y)\|)\leq \psi(\|x - f(x)\|,\|y-g(y)\|) - \phi_1(\|x-y\|)
\end{equation}
for all $x,y\in C$, where
\begin{enumerate}
\item[$i)$] $\phi$, $\phi_1$ are two altering distance functions;
\item[$ii$)] $\psi:[0,\infty)\times [0,\infty) \to [0,\infty)$ is a function such that $\psi$ is continuous at $(0,0)$, $\psi(0,0)=0$ and $\psi(t_1,t_2) < \phi(t_1)~~ \mbox{or}~~ \phi(t_2)$ if $t_1 >0$ or $t_2 >0$ or both.
\end{enumerate}
Then $f$ and $g$ have a unique common fixed point.
\end{theorem}
\begin{proof}
Let $x_0 \in C$ be arbitrary but fixed. We construct the sequence $\{x_n\}$ where $x_{2k}= g(x_{2k-1})$ for all $ k\in \mathbb{N}$ and $x_{2k+1}= f(x_{2k})$ for all $k\in \mathbb{N} \cup \{0\}$. Then we  consider the sequence $\{s_n\}$ of real numbers, where $s_n =\|x_n-x_{n+1}\|$ for all $n\in \mathbb{N}$. 
%If $x_n=x_{n+1}$ for finitely many $n$, then we can exclude those $x_n$ from $\{s_n\}$ and so we can assume that no two consecutive terms  of $\{x_n\}$ are equal. If $x_n= x_{n+1}$ for infinitely many $n$ and  $x_n\neq x_{n+1}$ for finitely many $n$, then it is an easy task to check that $\{s_n\}$ converges to $0$. Finally if  $x_n=x_{n+1}$ for infinitely many $n$ and  $x_n\neq x_{n+1}$  for infinitely many $n$, then b above discussion we see that $\{s_n\}$ converges to $0$.

First we assume that $n$ is odd. Then putting $x=x_{n+1}$, $y=x_n$ in Equation \eqref{e1} we get
\begin{align}
\phi(\|f(x_{n+1})-g(x_n)\|) &\leq \psi(\|x_{n+1} - f(x_{n+1})\|,\|x_n-g(x_n)\|) - \phi_1(\|x_{n+1}-x_n\|)\nonumber\\
  \label{e2} \Rightarrow \phi(\|x_{n+2}-x_{n+1}\|) &\leq \psi(\|x_{n+1} - x_{n+2}\|,\|x_n-x_{n+1}\|) - \phi_1(\|x_{n+1}-x_n\|).
\end{align}
If $\|x_{n+1} - x_{n+2}\|=\|x_n-x_{n+1}\|=0$, then we have $s_{n+1}=s_n$, i.e., $s_{n+1}\leq s_n$  otherwise we have $$\psi(\|x_{n+1} - x_{n+2}\|,\|x_n-x_{n+1}\|) < \phi(\|x_{n+1} - x_{n+2}\|)~\mbox{or}~\phi(\|x_n-x_{n+1}\|).$$
If  $\psi(\|x_{n+1} - x_{n+2}\|,\|x_n-x_{n+1}\|) < \phi(\|x_{n+1} - x_{n+2}\|)$, then from Equation \eqref{e2} we get $$\phi(\|x_{n+2}-x_{n+1}\|) < \phi(\|x_{n+1} - x_{n+2}\|) - \phi_1(\|x_{n+1}-x_n\|),$$ which is not possible. So we must have,  $$\psi(\|x_{n+1} - x_{n+2}\|,\|x_n-x_{n+1}\|) < \phi(\|x_n - x_{n+1}\|).$$ Then from Equation \eqref{e2} we have 
\begin{align*}
\phi(\|x_{n+2}-x_{n+1}\|) &<  \phi(\|x_n - x_{n+1}\|) - \phi_1(\|x_{n+1}-x_n\|)\\
\Rightarrow \|x_{n+2}-x_{n+1}\| &< \|x_n - x_{n+1}\|\\
\Rightarrow s_{n+1} &< s_n, ~\mbox{i.e.},~ s_{n+1} \leq s_n.
\end{align*}
Again if $n$ be even, then putting $x=x_n$ and $y=x_{n+1}$ in equation \eqref{e1} we get
\begin{align}
\phi(\|f(x_n)-g(x_{n+1})\|) &\leq \psi(\|x_n - f(x_n)\|,\|x_{n+1}-g(x_{n+1})\|) - \phi_1(\|x_n-x_{n+1}\|)\nonumber\\
  \label{e3} \Rightarrow \phi(\|x_{n+1}-x_{n+2}\|) &\leq \psi(\|x_n - x_{n+1}\|,\|x_{n+1}-x_{n+2}\|) - \phi_1(\|x_{n+1}-x_n\|).
\end{align}
Now using he property that $\psi(0,0)=0$ and  $\psi(t_1,t_2) < \phi(t_1)~~ \mbox{or}~~ \phi(t_2)$ if $t_1>0$ or $t_2>0$ or both and proceeding as in the previous case, we can obtain from Equation \eqref{e3} that $s_{n+1} \leq s_n.$

Thus we have $s_{n+1} \leq s_n$ for all $n \in \mathbb{N}$. Therefore, $\{s_n\}$ is a monotone decreasing sequence of non-negative real numbers and hence convergent to some $a\in \mathbb{R}$, say.

 Next, we show that $a=0$.
For any even $n\in \mathbb{N}$, putting $x=x_n$ and $y=x_{n+1}$ in Equation \eqref{e1} we get
\begin{equation} \label{e4}
\phi(\|x_{n+1}-x_{n+2}\|) \leq \psi (\|x_n - x_{n+1}\|, \|x_{n+1}-x_{n+2}\|)-\phi_1(\|x_n - x_{n+1}\|).
\end{equation}

If $\|x_n - x_{n+1}\|=\|x_{n+1}-x_{n+2}\|=0$, then we have $\psi (\|x_n - x_{n+1}\|, \|x_{n+1}-x_{n+2}\|)=0$ and so from Equation \eqref{e4} we get
$$\phi(\|x_{n+1}-x_{n+2}\|)+\phi_1(\|x_n - x_{n+1}\|)\leq 0.$$
Letting $n \to \infty$ in both sides of above equation we get $$\phi(a) + \phi_1(a) \leq 0$$ and this gives $a=0$.

If $\|x_n - x_{n+1}\| \neq 0$ or $\|x_{n+1}-x_{n+2}\| \neq 0$ or both, then $$\psi (\|x_n - x_{n+1}\|, \|x_{n+1}-x_{n+2}\|) < \phi(\|x_n - x_{n+1}\|)~\mbox{or}~ \phi(\|x_{n+1}-x_{n+2}\|).$$

If $\psi (\|x_n - x_{n+1}\|, \|x_{n+1}-x_{n+2}\|) < \phi(\|x_{n+1}-x_{n+2}\|)$, then from Equation \eqref{e4} we get 
$$\phi(\|x_{n+1}-x_{n+2}\|) < \phi(\|x_{n+1}-x_{n+2}\|)-\phi_1(\|x_n - x_{n+1}\|),$$ which is not possible, so we must have $$\psi (\|x_n - x_{n+1}\|, \|x_{n+1}-x_{n+2}\|) < \phi(\|x_n-x_{n+1}\|).$$
Then from Equation \eqref{e4} we have 
\begin{equation} \label{e5}
\phi(\|x_{n+1}-x_{n+2}\|) < \phi(\|x_n-x_{n+1}\|)-\phi_1(\|x_n - x_{n+1}\|).
\end{equation}
Now, letting $n \to \infty$ in both sides of Equation \eqref{e5} and using the continuity of $\phi$, $\phi_1$ we get 
\begin{align*}
\phi(a)& \leq \phi(a) - \phi_1(a)\\
\Rightarrow \phi_1(a)&=0 \Rightarrow a =0.
\end{align*}
Therefore, $$\displaystyle \lim_{n\to \infty}s_n=0.$$
%Now if $x_n=x_{n+1}$ for finitely many $n$, then we can exclude those $x_n$ from $\{x_n\}$ and  we can assume that no two consecutive terms  of $\{x_n\}$ are equal, then by above discussion we have $$\displaystyle \lim_{n\to \infty}s_n=0.$$
%If $x_n= x_{n+1}$ for infinitely many $n$ and  $x_n\neq x_{n+1}$ for finitely many $n$, then it is an easy task to check that $$\displaystyle \lim_{n\to \infty}s_n=0.$$
%%Finally if  $x_n=x_{n+1}$ for infinitely many $n$ and  $x_n\neq x_{n+1}$  for infinitely many $n$, then  also we see that $\{s_n\}$ converges to $0$. 
%
%Thus we have, $$\displaystyle \lim_{n\to \infty}\|x_n - x_{n+1}\|=0.$$
Next, we show that $\{x_n\}$ is a Cauchy sequence. Let $n,m \in \mathbb{N}$ be arbitrary.
%If $x_n = x_m$ for some $n,m$, then we can exclude that $x_n, x_m$ from $\{x_n\}$. So without loss of generality we may assume that  $x_n \neq x_m$ for all $n,m$. 
Then for even $n$ and odd $m$, putting $x=x_n$  and $y=x_m$ in Equation \eqref{e1} we get
\begin{align}
\phi(\|f(x_n)-g(x_m)\|) &\leq \psi(\|x_n -f(x_n)\|,\|x_m-g(x_m)\|) - \phi_1(\|x_n - x_m\|) \nonumber\\
\Rightarrow \phi(\|x_{n+1}-x_{m+1}\|) &\leq \psi(\|x_n -x_{n+1}\|,\|x_m-x_{m+1}\|) - \phi_1(\|x_n - x_m\|) \nonumber\\
\Rightarrow \phi(\|x_{n+1}-x_{m+1}\|) &\leq \psi(\|x_n -x_{n+1}\|,\|x_m-x_{m+1}\|) \label{e6}
\end{align}
Since $\psi$ is continuous at $(0,0)$, so for any $\epsilon >0$, there is a $\delta>0$ such that $$\psi(\|x_n -x_{n+1})\|,\|x_m-x_{m+1}\|)<\phi\left(\frac{\epsilon}{2}\right)$$ whenever $\|x_n -x_{n+1}\|^2 + \|x_m-x_{m+1}\|^2 < \delta^2$.

Again since, $\displaystyle \lim_{n\to \infty}\|x_n-x_{n+1}\|=0=\displaystyle \lim_{m\to \infty}\|x_m-x_{m+1}\|=0$, for the above $\delta>0$, there is a natural number $N_1$ such that
$$\|x_n-x_{n+1}\| < \frac{\delta}{\sqrt{2}}~~,~~\|x_m-x_{m+1}\| < \frac{\delta}{\sqrt{2}}$$ for all $n,~m \geq N_1.$ Therefore, for all $n,~m \geq N_1$ we have
\begin{align*}
\|x_n-x_{n+1}\| ^2 + \|x_m-x_{m+1}\|^2 &< \frac{\delta^2}{2} + \frac{\delta^2}{2} =\delta^2 \\
\Rightarrow \psi(\|x_n -x_{n+1}\|,\|x_m-x_{m+1}\|)&<\phi\left(\frac{\epsilon}{2}\right)\\
\Rightarrow \phi(\|x_{n+1}-x_{m+1}\|)&< \phi\left(\frac{\epsilon}{2}\right)~~ \Big[\mbox{using Equation \ref{e6}}\Big]\\
\Rightarrow \|x_{n+1}-x_{m+1}\| &< \frac{\epsilon}{2}
\end{align*}
Thus, $\|x_{n}-x_{m}\| < \frac{\epsilon}{2}$ for all $n,~m~ \geq N_1 +1$. Similarly if $n$ is odd and $m$ is even, we get a natural number $N_2$ such that $$\|x_{n}-x_{m}\| < \frac{\epsilon}{2}$$ for all $n,~m~ \geq N_2 +1.$ 
Choose $N= \max \{N_1 + 1, N_2 + 1\}$. Therefore, $$\|x_n-x_m\|< \frac{\epsilon}{2}$$ for all $n,~m \geq N$ with any one of $n,~m$ is even and the other is odd. 

If $n,~m$  both are even, then we have 
\begin{align*}
\|x_{n}-x_{m}\| &\leq  \|x_{n}-x_{n+1}\| + \|x_{n+1}-x_{m}\|\\
&< \frac{\epsilon}{2} + \frac{\epsilon}{2}=\epsilon
\end{align*}
for all $n,~m \geq N$. Again if $n,~m$  both are odd, then we can similarly show that $$\|x_{n}-x_{m}\|< \epsilon$$ for all $n,~m \geq N$. Hence $$\|x_{n}-x_{m}\|< \epsilon$$ for all $n,~m \geq N$. Thus $\{x_n\}$ is a Cauchy sequence in the closed subset $C$ of the Banach space $X$ and so convergent to some $z\in C$. Therefore the two subsequences $\{x_{2n}\}$ and $\{x_{2n-1}\}$ converge to $z$. Then by the continuity of $f$ and $g$ we have $$f(x_{2n}) \to f(z)~~\mbox{and}~~ g(x_{2n+1}) \to g(z)~~\mbox{as}~~ n \to \infty,$$i.e., $$x_{2n+1} \to f(z)~~\mbox{and}~~ x_{2n+2} \to g(z)~~\mbox{as}~~ n \to \infty.$$ So by the uniqueness of limit of the sequence $\{x_n\}$, we have $f(z)=z=g(z)$, i.e., $z$ is a common fixed point of $f$ and $g$.

Finally, we check the uniqueness of the common fixed point $z$. To check this, let $z_1(z_1\neq z)$ be another common fixed point of $f$ and $g$. Then from Equation \eqref{e1} we have
\begin{align*}
&\phi(\|f(z) - g(z_1)\|) \leq \psi(\|z - f(z)\|, \|z_1 - g(z_1)\|) - \phi_1(\|z-z_1\|)\\
&\Rightarrow \phi(\|f(z) - g(z_1)\|) \leq \psi(0,0) - \phi_1(\|z-z_1\|)\\
&\Rightarrow \phi(\|f(z) - g(z_1)\|) + \phi_1(\|z-z_1\|) \leq 0,
\end{align*}
which leads to a contradiction and so $z$ is the unique common fixed point of $f$ and $g$.
\end{proof}
In the next theorem of this section we relax the continuity of $f$ and $g$ in Theorem \ref{t1} to check the existence of common fixed point of $f$ and $g$. To give the guarantee of common fixed point of $f$ and $g$, we  prove the following theorem by  making a slide change on the control function $\psi$.
\begin{theorem} \label{t2}
Let $X$ be a Banach space and $C$ be a closed subset of $X$. Let $f,g :C \to C$ be two functions such that
\begin{equation} \label{e7}
\phi(\|f(x)-g(y)\|)\leq \psi(\|x - f(x)\|,\|y-g(y)\|) - \phi_1(\|x-y\|)
\end{equation}
for all $x,y\in C$, where
\begin{enumerate}
\item[$i)$] $\phi$, $\phi_1$ are two altering distance functions;
\item[$ii)$] $\psi:[0,\infty)\times [0,\infty) \to [0,\infty)$ is a continuous function such that $\psi(t_1,t_2) \leq \phi(t_1)~~ \mbox{and}~~ \phi(t_1)$ for all $t_1,t_2 \in [0,\infty)$.
\end{enumerate}
Then $f$ and $g$ has a unique common fixed point.
\end{theorem} 
\begin{proof}
Let $x_0 \in C$ be arbitrary and consider the sequence $\{x_n\}$ where $x_{2k}= g(x_{2k-1})$ $k\in \mathbb{N}$
and $x_{2k+1}= f(x_{2k})$ for all $k\in \mathbb{N} \cup \{0\}$. Then continuing as in Theorem \ref{t1}, we can show that $\{x_n\}$ is Cauchy sequence and hence convergent to some $z\in C$. Therefore the two subsequences $\{x_{2n}\}$ and $\{x_{2n+1}\}$  of $\{x_n\}$ also converge to $z$. Finally putting $x=x_{2n} $ and $y=z$ in Equation \ref{e7} we get
\begin{align}
\phi(\|f(x_{2n})-g(z)\|) &\leq \psi(\|x_{2n} - f(x_{2n})\|,\|z-g(z)\|) - \phi_1(\|x_{2n}-z\|)\nonumber\\
\Rightarrow \phi(\|x_{2n+1}-g(z)\|) &\leq \psi(\|x_{2n} - x_{2n+1}\|\|z-g(z)\|) - \phi_1(\|x_{2n}-z\|) \label{e8}.
\end{align} 
Letting $n \to \infty$ in both sides of Equation \ref{e8} and using the continuity of $\phi,~\psi,~\phi_1$ we get
\begin{align*}
\phi(\|z-g(z)\|) & \leq \psi(0,\|z-g(z)\|) - \phi_1(0)\\
\Rightarrow \phi(\|z-g(z)\|)  & \leq \phi(0)\\
\Rightarrow g(z)=z.
\end{align*}
In a similar manner, we can show that $f(z)=z$. Thus, $z$ is a common fixed point of $f$ and $g$. The uniqueness of the common fixed point $z$ being similar to that of Theorem \ref{t1}, is omitted.
\end{proof}
From Theorem \ref{t1} we have the following corollaries:
\begin{corollary}
Let $X$ be a Banach space and $C$ be a closed subset of $X$. Let $f,g :C \to C$ be two continuous functions such that
$$\phi(\|f(x)-g(y)\|)\leq \phi\left(\max\{\alpha\|x - f(x)\|,\alpha\|y-g(y)\|\}\right) - \phi_1(\|x-y\|)$$
for all $x,y\in C$ with $x\neq y$, where $\phi$, $\phi_1$ are two altering distance functions and $\alpha$ is a real number with $\alpha<1$.
Then $f$ and $g$ has a unique common fixed point.
\end{corollary}
\begin{proof}
If we take $\psi(t_1,t_2) = \phi\left(\max\{\alpha t_1,\alpha t_2\}\right) $ for all $t_1,~t_2 \in [0,\infty)$ in Theorem \ref{t1}, then the claim of the corollary holds.
\end{proof}
\begin{corollary}
Let $X$ be a Banach space and $C$ be a closed subset of $X$. Let $f :C \to C$ be a continuous function such that
$$\phi(\|f(x)-f(y)\|)\leq \phi\left(\max\{\alpha\|x - f(x)\|,\alpha\|y-f(y)\|\}\right) - \phi_1(\|x-y\|)$$
for all $x,y\in C$ with $x\neq y$, where $\phi$, $\phi_1$ are two altering distance functions and $\alpha$ is a real number with $\alpha<1$.
Then $f$  has a unique fixed point.
\end{corollary}
\begin{proof}
If we take $\psi(t_1,t_2) = \phi\left(\max\{\alpha t_1,\alpha t_2\}\right) $ for all $t_1,~t_2 \in [0,\infty)$ and $f=g$ in Theorem \ref{t1}, then we are done.
\end{proof}
%\begin{corollary}
%Let $X$ be a Banach space and $C$ be a closed subset of $X$. Let $f :C \to C$ be a continuous function such that
%$$\phi(\|f(x)-f(y)\|)< \phi\left(\|x - f(x)\|\right) - \phi_1(\|x-y\|)$$
%for all $x,y\in C$ with $x\neq y$, where $\phi$, $\phi_1$ are two altering distance functions.
%Then $f$  has a unique fixed point.
%\end{corollary}
%\begin{proof}
%The proof of the corollary will be completed if we take $\psi(t_1,t_2) =\phi(t_1)$ for all $t_1,~t_2 \in [0,\infty)$ and $f=g$ in Theorem \ref{t1}.
%\end{proof}
Now we present an example in support of Theorem \ref{t1}.
\begin{example}
Consider the Banach space $l^{\infty}$ endowed with sup norm and take $C=\{e_0, e_7, e_8, e_9, \dots\}$ where $e_0$ is the sequence having all terms $0$, $e_i$ is the sequence whose $i$-th term is $\frac{1}{2^i}$ and all other terms $0$. Then clearly $C$ is closed subset of $l^{\infty}$. Define two functions $f,g : C \to C$ by $$f(x) = e_0$$ for all $x\in C$ and $$
g(x)=
\begin{cases}
e_0, ~~\mbox{if}~~ x=e_0;\\
e_{i+5}, ~~\mbox{if}~~ x=e_i,~i\geq 7.
\end{cases}$$
Define a function $\psi:[0,\infty)\times [0,\infty) \to [0,\infty) $  by
$$\psi(t_1,t_2) = \begin{cases}
\frac{t_1 + t_2}{20}, ~~\mbox{if}~~t_1,t_2\leq \frac{1}{10};\\
\frac{1}{100},~~\mbox{elsewhere}.
\end{cases}$$
Also define two functions $\phi, \phi_1: [0,\infty) \to [0,\infty)$ by
$$\phi(t)=\begin{cases}
\frac{2t}{10}, ~~\mbox{if}~~ t\leq \frac{1}{10};\\
\frac{2}{100}, ~~\mbox{if}~~ t>\frac{1}{10},
\end{cases}$$
and $$\phi_1(t)=\frac{t}{10\times 2^4}$$ for all $t \in [0,\infty).$ Then, one can easily verify that $f,g$ are two continious functions, and $\phi,~\phi_1$ are two altering distance functions, $\psi$ is  continuous at $(0,0)$  with $\psi(0,0)=0$ and $\psi(t_1,t_2) < \phi(t_1) ~ \mbox{or}~ \phi(t_2)$ if $t_1>0$ or $t_2>0$ or both. Now let $x,~y \in [0,\infty)$ be arbitrary. Then the following cases will arise:

\textit{\textbf{Case 1:}}
$x=e_0,~ y=e_7$ for some $i\geq 7$. Then $$f(e_0)= e_0,~ g(y)= e_{i+5}.$$ Therefore, 
\begin{align*}
\phi(\|f(x) - g(y)\|) + \phi_1(\|x-y\|) &= \phi(\|e_{i+5}\|) + \phi_1(\|e_i\|)\\
&= \phi(\frac{1}{2^{i+5}}) + \phi_1 (\frac{1}{2^i})\\
&= \frac{1}{10\times 2^{i+4}} + \frac{1}{10\times 2^{i+4}} = \frac{1}{10\times 2^{i+3}},
\end{align*}
whereas, 
\begin{align*}
\psi(\|x-f(x)\|,\|y-g(y)\|) &= \psi(\|e_0\|,\|e_i - e_{i+5}\|)\\
&= \psi(0, \frac{1}{2^i})\\
&= \frac{1}{10 \times 2^{i+1}}.
\end{align*}
Therefore, $$\phi(\|f(x) - g(y)\|) \leq \psi(\|x-f(x)\|,\|y-g(y)\|) - \phi_1(\|x-y\|).$$

\textit{\textbf{Case 2:}} $x=e_i$, for some $i \geq 7$ and $y=e_0$. Then $f(x)=e_0=g(y)$. Therefore,
\begin{align*}
\phi(\|f(x) - g(y)\|) + \phi_1(\|x-y\|) &= \frac{1}{10\times 2^{i+4}},
\end{align*}
and
\begin{align*}
\psi(\|x-f(x)\|,\|y-g(y)\|) &= \psi(\|e_i\|,\|e_0\|)\\
&=\psi(\frac{1}{2^i},0)\\
&= \frac{1}{10 \times 2^{i+1}}.
\end{align*}
Then, clearly we have $$\phi(\|f(x) - g(y)\|) \leq \psi(\|x-f(x)\|,\|y-g(y)\|) - \phi_1(\|x-y\|).$$
\textit{\textbf{Case 3:}} $x=e_i$, $y=e_j$ for some $i,~j \geq 7$. Then, we have $$f(x)= e_{0},~ g(y)= e_{j+5}.$$ Now 
\begin{align*}
\phi(\|f(x) - g(y)\|)&= \phi(\|e_{j+5}\|)\\
&=\frac{1}{10\times 2^{j+4}}.
\end{align*}
$$
\phi(\|x - y\|)=
\begin{cases}
\frac{1}{10\times 2^{j+4}},~~\mbox{if}~~ i\geq j;\\
\frac{1}{10\times 2^{i+4}},~~\mbox{if}~~ j\geq i.
\end{cases}$$
\begin{align*}
\psi(\|x-f(x)\|,\|y-g(y)\|) &=\psi(\|e_i\|,\|e_j - e_{j+5}\|)\\
&= \psi(\frac{1}{2^i},\frac{1}{2^j})\\
&= \frac{1}{10 \times 2^{i+1}}+\frac{1}{10 \times 2^{j+1}}.
\end{align*}
Therefore for $i\geq j$, we have
\begin{align*}
\phi(\|f(x)-g(y)\|) + \phi_1(\|x-y\|) &= \frac{1}{10\times 2^{j+4}} + \frac{1}{10\times 2^{j+4}}\\
&= \frac{1}{10\times 2^{j+3}}\\
&\leq \frac{1}{10\times 2^{j+1}} + \frac{1}{10\times 2^{i+1}}\\
&= \psi(\|x-f(x)\|,\|y-g(y)\|).
\end{align*}
Also for $j\geq i$, we have
\begin{align*}
\phi(\|f(x)-g(y)\|) + \phi_1(\|x-y\|) &= \frac{1}{10\times 2^{j+4}} + \frac{1}{10\times 2^{i+4}}\\
%&= \frac{1}{10\times 2^{j+3}}\\
&\leq \frac{1}{10\times 2^{j+1}} + \frac{1}{10\times 2^{i+1}}\\
&= \psi(\|x-f(x)\|,\|y-g(y)\|).
\end{align*}
Thus combining all the three cases we see that $$\phi(\|f(x)-g(y)\|)\leq \psi(\|x-f(x)\|,\|y-g(y)\|) - \phi_1(\|x-y\|)$$ for all $x,~y \in C$ with $x \neq y$. Therefore all the conditions of Theorem \ref{t1} are satisfied and so by that theorem $f$ and $g$ has a unique common fixed point in $C$. Indeed $e_0$ is the unique common fixed point $f$ and $g$.
\end{example}
\section{\bf Application to Non-linear Matrix Equations}
In this section, we study the following general non-linear matrix equations:
$$X=Q_1 \pm \displaystyle \sum_{i=1}^{m} {A_i}^*F(X)A_i,$$
$$X=Q_2 \pm \displaystyle \sum_{i=1}^{m} {A_i}^*G(X)A_i,$$
where $Q_1,Q_2$ is an $n \times n$ Hermitian positive definite matrix, and $A_1, A_2, \dots, A_n$ are arbitrary $n \times n$ matrices, and $F,G$ are two functions defined on the set of all Hermitian matrices into itself and they are continuous with respect to some suitable norm. We present sufficient conditions which will ensure positive definite common solution of the above pair of matrix equations. In the remaining portion, the symbols $\lambda (A),~\lambda^{+}(A)$  denotes a singular value of $A$ and the sum of all singular values of $A$ respectively. Also $\|.\|$ denotes the trace norm, i.e., $\|A\|=\lambda^{+} (A) $. Again for $A, B \in H(n)$, by $A \succeq B~(A \succ B)$ we mean $A-B$ is positive semi-definite (positive definite).
%$\bullet$ The symbol $M(n)$ stands for the set of all $n\times n$ matrices.\\
%$\bullet$ The symbol $H(n)$ stands for the set of all $n\times n$ Hermitian matrices.\\
%$\bullet$ The symbol $P(n)$ stands for the set of all $n\times n$ positive definite matrices.\\
%$\bullet$ The symbol $\overline{P(n)}$ stands for the set of all $n\times n$  positive semi-definite matrices.\\
%$\bullet$ The symbol $\lambda (A)$ denotes an singular value of $A$.\\
%$\bullet$ The symbol $\lambda^{+}(A)$ denotes the sum of all singular values of $A$.\\
%$\bullet$ The symbol $\|.\|$ denotes the trace norm, i.e., $\|A\|=\lambda^{+} (A) $.\\
%$\bullet$ For $A, B \in H(n)$, $A \succeq B~(A \succ B)$ means $A-B$ is positive semi-definite (positive        definite).

It is well known that the set $H(n)$  equipped the trace norm is complete, for more details of this type of norm, see \cite {RR,BS,B3}.

Now we prove a theorem pointing to sufficient conditions for common solution to the non-linear matrix equations.
\begin{theorem} \label{t5}
Let us consider the following pair of non-linear matrix equations:
\begin{equation} \label{e13}
X=Q_1 + \displaystyle \sum_{i=1}^{m} {A_i}^*F(X)A_i
\end{equation}
\begin{equation} \label{e14}
X=Q_2 + \displaystyle \sum_{i=1}^{m} {A_i}^*G(X)A_i
\end{equation}
where $Q_1,Q_2 \in P(n)$, $A_i \in M(n)$, and $F,G$ are two continuous functions with respect to the trace norm of $H(n)$. Let, for any $X \in H(n)$ with $\|X\| \leq a$, $\lambda(F(X)),\lambda(G(X)) \leq k_1$  for all $\lambda(F(X)),\lambda(G(X))$, for some $a,k_1 \in \mathbb{R}$ and assume  that
\begin{enumerate}
\item[$i)$] $\displaystyle \sum_{i=1}^{m} \|{A_i}^*\| \|A_i\| = k ,$ and $\|Q_1\|, \|Q_2\|\leq a - k k_1 n$;
\item[$ii)$] for any $X \in H(n)$ with $\|X\| \leq a$, either  $\displaystyle \sum_{i=1}^{m} {A_i}^*F(X)A_i \succeq  O$ \\or $\displaystyle \sum_{i=1}^{m} {A_i}^*G(X)A_i \succeq O$;
\item[$iii)$] $2kk_1 + \lambda(Q_1 - Q_2) \leq \frac{1}{n+1}\max \Bigg \{\Big|\lambda^{+}\Big(\displaystyle \sum_{i=1}^{m} {A_i}^*F(X)A_i\Big) - \lambda^{+}(X- Q_1)\Big|, \\ \Big|\lambda^{+}\Big(\displaystyle \sum_{i=1}^{m} {A_i}^*G(Y)A_i\Big) - \lambda^{+}(Y- Q_2)\Big|\Bigg \} -\alpha \lambda^{+}(X-Y) $
\end{enumerate}
%\begin{equation*}
%\begin{split}
%i) ~\displaystyle \sum_{i=1}^{m} \|{A_i}^*\| \|A_i\| = k < \frac{1}{k_1}, \mbox{and}~ \|Q_1\|\leq a(1-kk_1), \|Q_2\|\leq a(1-kk_1, for some real number a)  \\
%ii) ~ \\
%iii)~  2a + \lambda(Q_1 - Q_2) \leq \max \Bigg \{\Big|\lambda^{+}\left(\displaystyle \sum_{i=1}^{m} {A_i}^*F(X)A_i\right) - \lambda^{+}(X- Q_1)\Big|, \\
% \Big|\lambda^{+}\left(\displaystyle \sum_{i=1}^{m} {A_i}^*G(Y)A_i\right) - \lambda^{+}(Y- Q_2)\Big|\Bigg \} -\alpha \lambda^{+}(X-Y)
%\end{split}
%\end{equation*}
for all $X,Y \in H(n)$ with $\|X\|, \|Y\| \leq a$, and for any $\lambda(Q_1 - Q_2)$,  where $\alpha$ is a positive real number (may be very small). Then the pair of matrix equations given by Equations \eqref{e13} and \eqref{e14} have a unique common positive definite solution $\hat{X}$ with $\|\hat{X}\| \leq a$.
% Further, the solution is exhibited as $$\hat{X} = $
\end{theorem}
\begin{proof}
Let us consider the set $\mathcal{C}=\{X \in H(n) : \|X\|\leq a\}$. Then, one can easily verify that $C$ is a closed subset of $H(n)$. First observe that, assumption $(ii)$ ensures that any common solution of Equations \eqref{e13} and \eqref{e14} in $\mathcal{C}$ must be positive definite.

Now for any $X\in \mathcal{C}$, we have
\begin{align}
\|Q_1 + \displaystyle \sum_{i=1}^{m} {A_i}^*F(X)A_i\| &\leq \|Q_1\| + \|\displaystyle \sum_{i=1}^{m} {A_i}^*F(X)A_i\|  \nonumber\\
&\leq \|Q_1\| + \displaystyle \sum_{i=1}^{m} \|{A_i}^*\| \|A_i\| \|F(X)\| \nonumber\\
&= \|Q_1\| + k  \|F(X)\| \label{ee1}
%&\leq a(1-kk_1) + k k_1a = a.
\end{align}
Now we have $\lambda(F(X)) \leq k_1$ for all $\lambda(F(X))$, so adding over all $\lambda(F(X))$ we get $$\|F(X)\| \leq nk_1.$$
Therefore, from Equation \ref{ee1} we have 
\begin{align*}
\|Q_1 + \displaystyle \sum_{i=1}^{m} {A_i}^*F(X)A_i\| &\leq \|Q_1\| + k k_1 n\\
&\leq a - k k_1 n + k k_1 n = a.
\end{align*}
Similarly for any $X\in \mathcal{C}$, we can show that 
$$\|Q_2 + \displaystyle \sum_{i=1}^{m} {A_i}^*G(X)A_i\| \leq a.$$ 
Next, we define two functions $f,g:\mathcal{C} \to \mathcal{C}$ by 
\begin{equation} \label{e15}
f(X)=Q_1 + \displaystyle \sum_{i=1}^{m} {A_i}^*F(X)A_i
\end{equation} 
and 
\begin{equation} \label{e16}
g(X)=Q_2 + \displaystyle \sum_{i=1}^{m} {A_i}^*G(X)A_i
\end{equation}
for all $X\in \mathcal{C}$. From the above discussion it can be easily verified that $f,~g$ are well defined on $\mathcal{C}$.  As a consequence, finding  common solution of Equations \eqref{e13} and \eqref{e14} is equivalent to finding common fixed point of $f$ and $g$.

Now for any $X,Y \in \mathcal{C}$, we have
\begin{align*}
\|f(X) -g(Y)\| &= \|Q_1 + \displaystyle \sum_{i=1}^{m} {A_i}^*F(X)A_i -Q_2 - \displaystyle \sum_{i=1}^{m} {A_i}^*G(Y)A_i \|\\
& \leq \|Q_1 - Q_2\| + \|\displaystyle \sum_{i=1}^{m}{A_i}^*F(X)A_i - \displaystyle \sum_{i=1}^{m} {A_i}^*G(Y)A_i\|\\
& \leq \|Q_1 - Q_2\| + \displaystyle \sum_{i=1}^{m}\|{A_i}^*F(X)A_i - {A_i}^*G(Y)A_i\|\\
& \leq \|Q_1 - Q_2\| + \displaystyle \sum_{i=1}^{m}\|{A_i}^*\| \|A_i\| \|F(X) - G(Y)\|\\
& \leq \|Q_1 - Q_2\| + k(\|F(X)\| + \|G(Y)\|)\\
& \leq \|Q_1 - Q_2\| + k(k_1n + k_1n)\\
%& \leq \|Q_1 - Q_2\| + k(k_1 a +k_1 a)\\
& = \|Q_1 - Q_2\| + 2kk_1n.
\end{align*}
Thus for any $X,Y \in \mathcal{C}$, we get
\begin{equation} \label{e17}
\|f(X) -g(Y)\| < \|Q_1 - Q_2\| + 2kk_1n.
\end{equation}
For some fixed $X,Y\in \mathcal{C}$, if $\max \Bigg \{\Big|\lambda^{+}\left(\displaystyle \sum_{i=1}^{m} {A_i}^*F(X)A_i\right) - \lambda^{+}(X- Q_1)\Big|,\Big|\lambda^{+}\left(\displaystyle \sum_{i=1}^{m} {A_i}^*G(Y)A_i\right) - \lambda^{+}(Y- Q_2)\Big|\Bigg \}$ = $\Big|\lambda^{+}\left(\displaystyle \sum_{i=1}^{m} {A_i}^*F(X)A_i\right) - \lambda^{+}(X- Q_1)\Big|$, then by given condition $(ii)$ we have 
\begin{align*}
2kk_1 + \lambda(Q_1 - Q_2) &\leq \frac{1}{n+1} \Big|\lambda^{+}\left(\displaystyle \sum_{i=1}^{m} {A_i}^*F(X)A_i\right) - \lambda^{+}(X- Q_1)\Big| - \alpha \lambda^{+}(X-Y)\\
\Rightarrow 2kk_1 + \lambda(Q_1 - Q_2) &\leq \frac{1}{n+1}\Big| \| \displaystyle \sum_{i=1}^{m}{A_i}^*F(X)A_i\| - \|X- Q_1\|\Big |  - \alpha \|X-Y\|\\
&\leq \frac{1}{n+1}\| \displaystyle \sum_{i=1}^{m}{A_i}^*F(X)A_i - X+ Q_1\|   - \alpha \|X-Y\|\\
&= \frac{1}{n+1}\|f(X) -X \| - \alpha \|X-Y\|.
\end{align*}
The above equation holds for every singular values of $(Q_1 - Q_2)$, so adding over all $n$ singular values of $(Q_1 - Q_2)$ we have
\begin{equation} \label{e18}
\|Q_1 - Q_2\| + 2kk_1n \leq \frac{n}{n+1}\|f(X) -X \| - n\alpha \|X-Y\|.
\end{equation}
Therefore, from Equation \eqref{e17} we have $$\|f(X) -g(Y)\| \leq \frac{n}{n+1}\|f(X) -X \| - n\alpha \|X-Y\|.$$
Similarly for some fixed $X,Y\in \mathcal{C}$ with $X\neq Y$, if $\max \Bigg \{|\lambda^{+}\left(\displaystyle \sum_{i=1}^{m} {A_i}^*F(X)A_i\right) - \lambda^{+}(X- Q_1)|,|\lambda^{+}\left(\displaystyle \sum_{i=1}^{m} {A_i}^*G(Y)A_i\right) - \lambda^{+}(Y- Q_2)|\Bigg \}$ = $|\lambda^{+}\left(\displaystyle \sum_{i=1}^{m} {A_i}^*G(Y)A_i\right) - \lambda^{+}(Y- Q_2)|$, then we can show that $$\|f(X) -g(Y)\| \leq \frac{n}{n+1}\|g(Y) -Y \| - n\alpha \|X-Y\|.$$
Thus for any $X,Y\in \mathcal{C}$, we see that 
\begin{equation}\label{e19}
\|f(X) -g(Y)\| \leq \max\Big\{\frac{n}{n+1}\|f(X) -X \|,\frac{n}{n+1}\|g(Y) -Y \|\Big\} - n\alpha \|X-Y\|.
\end{equation} Let us now define two functions $\phi, \phi_1 : [0,\infty) \to [0,\infty)$ by
$$\phi(t)= t, ~~\mbox{and}~~ \phi_1(t)=n\alpha t$$ for all $t \in [0,\infty)$. Therefore, $\phi, \phi_1$ are two altering distance functions.
Again define a function $\psi :[0,\infty) \times [0,\infty) \to [0,\infty)$ by $$\psi(t_1,t_2) = \phi\Big(\max\Big\{\frac{n}{n+1} t_1,\frac{n}{n+1} t_2\Big\}\Big)$$ for all $t_1, t_2 \in [0,\infty)$. Then, $\psi$ is a continuous function, and $$\psi(0,0) = 0~\mbox{and}~\psi(t_1,t_2)< \phi(t_1) ~~ \mbox{or}~~ \phi(t_2)$$ if $t_1>0$ or $t_2>0$ or both. Utilizing Equation \eqref{e19} and the formulations of $\phi, \phi_1, \psi$, we get $$\phi(\|f(X) -g(Y)\|) \leq \psi(\|f(X) -X \|,\|g(Y) -Y \|) - \phi_1( \|X-Y\|)$$ for all $X,Y \in \mathcal{C}$.  Also since $F,G$ are continuous, it follows that $f,g$ are also continuous on $\mathcal{C}$.

Owing to the above discussions, all the conditions of Theorem \ref{t1} are satisfied and hence $f, g$ has a unique common fixed point in $\mathcal{C}$, say $\hat{X}$. That is to say,  the pair of matrix equations \eqref{e13} and \eqref{e14} has a unique common positive definite solution and the solution is $\hat{X}$ with $\|\hat{X} \|\leq a$.
\end{proof}
From the above theorem we get the following the corollary.
\begin{corollary}
Consider the following non-linear matrix equation:
\begin{equation} \label{e20}
X=Q + \displaystyle \sum_{i=1}^{m} {A_i}^*F(X)A_i
\end{equation}
where $Q\in P(n)$, $A_i \in M(n)$, and $F$ is a continuous function with respect to the trace norm of $H(n)$. Let for any $X\in H(n)$ with $\|X\| \leq a$, $\lambda(F(X)) \leq k_1$  for all $\lambda(F(X)),$ for some  $a,k_1\in \mathbb{R}$ and assume that
\begin{enumerate}
\item[$i)$] $\displaystyle \sum_{i=1}^{m} \|{A_i}^*\| \|A_i\| = k,$ and $\|Q\|\leq a-kk_1n$;
\item[$ii)$] for any $X \in H(n)$ with $\|X\| \leq a$,  $\displaystyle \sum_{i=1}^{m} {A_i}^*F(X)A_i \succeq O $;
\item[$iii)$] $ \max \Bigg\{\Big|\lambda^{+}\left(\displaystyle \sum_{i=1}^{m} {A_i}^*F(X)A_i\right) - \lambda^{+}(X- Q)\Big|,\Big|\lambda^{+}\left(\displaystyle \sum_{i=1}^{m} {A_i}^*F(Y)A_i\right) - \lambda^{+}(Y- Q)\Big|\Bigg \} -\alpha \lambda^{+}(X-Y) \geq 2kk_1 (n+1), $
\end{enumerate}
for all $X,Y \in H(n)$ with $\|X\|,\|Y\| \leq a$, where $\alpha$ is a positive real number (may be very small). Then the  matrix equation given by Equation \eqref{e20} have a unique  positive definite solution $\hat{X}$ with $\|\hat{X}\| \leq a$.
\end{corollary}
\begin{proof}
The proof of the corollary simply follows from Theorem \ref{t5} by taking $Q_1=Q =Q_2$ and $F=G$.
\end{proof}
Next, we prove the following two theorems involving the same kind sufficient conditions relating common solution to non-linear matrix equations.
\begin{theorem}
Let us consider the following pair of non-linear matrix equations:
\begin{equation} \label{e21}
X=Q_1 - \displaystyle \sum_{i=1}^{m} {A_i}^*F(X)A_i
\end{equation}
\begin{equation} \label{e22}
X=Q_2 - \displaystyle \sum_{i=1}^{m} {A_i}^*G(X)A_i
\end{equation}
where $Q_1,Q_2 \in P(n),$ $A_i \in M(n)$, and $F,G$ are two continuous functions with respect to the trace norm of $H(n)$. Let, for any $X \in H(n)$ with $\|X\| \leq a$, $\lambda(F(X)),\lambda(G(X)) \leq k_1$  for all $\lambda(F(X)),\lambda(G(X)),$ for some $a,k_1 \in \mathbb{R}$ and assume  that
\begin{enumerate}
\item[$i)$] $\displaystyle \sum_{i=1}^{m} \|{A_i}^*\| \|A_i\| = k,$ and $\|Q_1\|,\|Q_2\|\leq a-kk_1n)$;
\item[$ii)$]  for any $X \in H(n)$ with $\|X\| \leq a$, either  $Q_1 \succ  \displaystyle \sum_{i=1}^{m} {A_i}^*F(X)A_i $ \\~or~ $ Q_2 \succ  \displaystyle \sum_{i=1}^{m} {A_i}^*G(X)A_i $;
\item[$iii)$] $2kk_1 + \lambda(Q_1 - Q_2) \leq \frac{1}{n+1}\max \Bigg \{\Big|\lambda^{+}\left(\displaystyle \sum_{i=1}^{m} {A_i}^*F(X)A_i\right) - \lambda^{+}(X+ Q_1)\Big|, \\ \Big|\lambda^{+}\left(\displaystyle \sum_{i=1}^{m} {A_i}^*G(Y)A_i\right) - \lambda^{+}(Y+ Q_2)\Big|\Bigg \} -\alpha \lambda^{+}(X-Y) $
\end{enumerate}
for all $X,Y \in H(n)$ with $\|X\|, \|Y\| \leq a$ and for any $\lambda(Q_1 - Q_2)$, where $\alpha$ is a positive real number (may be very small). Then the pair of matrix equations given by Equations \eqref{e21} and \eqref{e22} have a unique common positive definite solution $\hat{X}$ with $\|\hat{X}\| \leq a$.
\end{theorem}
\begin{proof}
The proof of this theorem is analogous to that of Theorem \ref{t5} and so omitted.
\end{proof}
\begin{theorem}
Let us consider the following pair of non-linear matrix equations:
\begin{equation} \label{e23}
X=Q_1 - \displaystyle \sum_{i=1}^{m} {A_i}^*F(X)A_i
\end{equation}
\begin{equation} \label{e24}
X=Q_2 + \displaystyle \sum_{i=1}^{m} {A_i}^*G(X)A_i
\end{equation}
where $Q_1,Q_2 \in P(n),$ $A_i \in M(n)$, and $F,G$ are two continuous functions with respect to the trace norm of $H(n)$. Let, for any $X \in H(n)$ with $\|X\| \leq a$, $\lambda(F(X)),\lambda(G(X)) \leq k_1$  for all $\lambda(F(X)),\lambda(G(X)),$ for some $a,k_1 \in \mathbb{R}$ and assume  that
\begin{enumerate}
\item[$i)$] $\displaystyle \sum_{i=1}^{m} \|{A_i}^*\| \|A_i\| = k,$ and $\|Q_1\|,\|Q_2\| \leq a-kk_1n$;
\item[$ii)$]  for any $X \in H(n)$ with $\|X\| \leq a$, either  $Q_1 \succ  \displaystyle \sum_{i=1}^{m} {A_i}^*F(X)A_i $\\ ~or~ $ \displaystyle \sum_{i=1}^{m} {A_i}^*G(X)A_i  \succeq O $;
\item[$iii)$] $2kk_1 + \lambda(Q_1 - Q_2) \leq \frac{1}{n+1}\max \Bigg \{\Big|\lambda^{+}\left(\displaystyle \sum_{i=1}^{m} {A_i}^*F(X)A_i\right) - \lambda^{+}(X+ Q_1)\Big|,\\ \Big|\lambda^{+}\left(\displaystyle \sum_{i=1}^{m} {A_i}^*G(Y)A_i\right) - \lambda^{+}(Y- Q_2)\Big|\Bigg \} -\alpha \lambda^{+}(X-Y) $
\end{enumerate}
for all $X,Y \in H(n)$ with $\|X\|, \|Y\| \leq a$ and for any $\lambda(Q_1 - Q_2) $, where $\alpha$ is a positive real number (may be very small). Then the pair of matrix equations given by Equations \eqref{e23} and \eqref{e24} have a unique common positive definite solution $\hat{X}$ with $\|\hat{X}\| \leq a$.
\end{theorem}
\begin{proof}
The proof of this theorem is also parallel to the proof of Theorem \ref{t5} and so we skip this proof also.
\end{proof}
\vskip.5cm\noindent{\bf Acknowledgements:}\\
The first named author is thankful to CSIR, New Delhi, India for their financial support. 
%\begin{thebibliography}{10}
\bibliographystyle{plain}
%\bibliography{bibliography_file}
%%\bibliographystyle{elsarticle-num}

\end{document}